\documentclass[11pt]{amsart}

\usepackage[english]{babel}
\usepackage{lineno,hyperref}
\usepackage[utf8]{inputenc}

\usepackage{amssymb}
\usepackage{amsmath}
\usepackage{mathrsfs}

\usepackage{comment}

\usepackage{color}

\newtheorem{theorem}{Theorem}[section]
\newtheorem{lemma}[theorem]{Lemma}
\newtheorem{proposition}[theorem]{Proposition}
\newtheorem{corollary}[theorem]{Corollary}
\newtheorem{definition}[theorem]{Definition}
\newtheorem{example}[theorem]{Example}

\newtheorem{remark}[theorem]{Remark}

\newcommand{\N}{\mathbb{N}}
\newcommand{\Z}{\mathbb{Z}}
\newcommand{\Q}{\mathbb{Q}}
\newcommand{\R}{\mathbb{R}}

\newcommand{\ns}[1]{\,\!^\ast#1} 
\newcommand{\sh}[1]{\,\!^\circ#1} 

\newcommand{\B}{\test({\dom})'}

\newcommand{\hR}{\ns{\R}} 
\newcommand{\fin}{\hR_{fin}}
\renewcommand{\Lambda}{\mathbb{X}}

\newcommand{\D}{\mathbb{D}}

\newcommand{\test}{\mathscr{D}_\Lambda}
\newcommand{\tests}{\mathscr{D}}
\newcommand{\supp}{\mathrm{supp\,}}

\newcommand{\grid}[1]{\mathbb{G}({#1})}

\renewcommand{\sim}{\approx}

\newcommand{\dom}{\Omega_{\Lambda}}

\newcommand{\norm}[1]{\Vert#1\Vert}
\newcommand{\bcf}{C^0_b}
\newcommand{\rad}{\mathbb{M}}
\newcommand{\prob}{\rad^{\mathbb{P}}}

\newcommand{\stdiv}{\mathrm{div}}
\renewcommand{\div}{\stdiv_{\Lambda}}

\newcommand{\grad}{\nabla_\Lambda}
\newcommand{\lap}{\Delta_\Lambda}

\newcommand{\ldual}{\langle}
\newcommand{\rdual}{\rangle_{\tests(\Omega)}}

\newcommand{\weaklys}{\stackrel{\star}{\rightharpoonup}}

\newcommand{\young}{\stackrel{Y}{\rightharpoonup}}

\begin{document}
		\title[Describing limits of integrable functions as grid functions]{Describing limits of integrable functions as grid functions of nonstandard analysis}
		
		\author{Emanuele Bottazzi}
		\address{University of Pavia, Italy}

		\begin{abstract}
			 In functional analysis, there are different notions of limit for a bounded sequence of $L^1$ functions.
			 Besides the pointwise limit, that does not always exist, the behaviour of a bounded sequence of $L^1$ functions can be described in terms of its weak-$\star$ limit or by introducing a measure-valued notion of limit in the sense of Young measures.
			 Working in Robinson's nonstandard analysis, we show that for every bounded sequence $\{z_n\}_{n \in \mathbb{N}}$ of $L^1$ functions there exists a function of a hyperfinite domain (i.e.\ a grid function) that represents both the weak-$\star$ and the Young measure limits of the sequence.
			 This result has relevant applications to the study of nonlinear PDEs. We discuss the example of an ill-posed forward-backward parabolic equation.
		\end{abstract}

\maketitle
\tableofcontents

\section{Introduction}
The lack of a nonlinear theory of distributions, first established by Schwartz \cite{schwartz}, poses some limitations in the study of nonlinear PDEs:
while some nonlinear problems can be solved by studying the limit of suitable regularized problems, other are ill-posed in the sense that they do not allow for solutions in the space of distributions.
For some of these problems, the notion of admissible solution can be meaningfully extended to include measure-valued solutions (we refer to \cite{evans nonlinear} for a theoretical discussion on the issue, and to \cite{smarri2, illposed, demoulini, matete, plotnikov, slemrod, smarrazzo} for some examples of measure valued solutions to ill-posed PDEs).
These measure-valued solutions are obtained as suitable limits of approximate solution in the presence of some estimates. For instance, a uniformly bounded sequence of integrable functions has a weak-$\star$ limit that corresponds to a Radon measure and has also limit in the sense of Young measures. For some PDEs, both limits must be considered in order to obtain a measure-valued solution. An example is studied in detail in \cite{smarrazzo}, see also Section \ref{sec 4} of this paper.

In order to overcome the absence of a nonlinear theory for distributions, some authors have embedded the space of distributions in a differential algebra with a good nonlinear theory.
A pioneer of this line of research is Colombeau, that in 1983 proposed an organic approach to a nonlinear theory of distributions \cite{colombeau 1983,colombeau 2001}. Colombeau's idea is to embed the distributions in a differential algebra of equivalence classes of smooth maps. This algebra allows for a good nonlinear theory via a canonical extension of classical operations.
Colombeau's approach has been met with interest and has proved to be a prolific field of research.

Colombeau algebras, however, lack some features with respect to more classical mathematical objects. For instance, according to Giordano, Kunziger and Vernaeve they do not yet have ``general existence theorems, comparable to the functional-analytic foundations of distribution theory'' \cite{giordano}. Another drawback is that the ring of scalars of the algebra is a non-Archimedean extension of $\R$ that however includes zero-divisors.

In order to improve on the first limitation, Giordano, Kunziger and Vernaeve introduced a new notion generalized functions, namely generalized smooth functions, that can be seen as a generalization of Colombeau functions to general domains and that allow for better set-theoretical properties \cite{giordano}.

The second drawback has been addressed by  
Todorov et al.\ with the introduction of algebras of asymptotic functions defined over a Robinson field of asymptotic numbers \cite{oberbuggenberg}. 
The algebras of asymptotic functions can be seen as generalized Colombeau algebras where the set of scalars is an algebraically closed field rather than a ring with zero divisors \cite{todorov}.
In this setting, it is possible to study generalized solutions to differential equations, and in particular to those with nonsmooth coefficient and distributional initial data \cite{nonsmo2, nonsmo1, todonew}.

The asymptotic functions are only one of the many algebras of generalized functions that can be defined in the setting of Robinson's nonstandard analysis.
Possibly the earliest results in this field are the proofs by Robinson that the distributions can be represented by smooth functions of nonstandard analysis and by polynomials of a hyperfinite degree \cite{nsa robinson,stroyan}.
Other algebras of nonstandard functions that are expressive enough for the representation of distributions have been studied in \cite{hyperfinite pinto,moto}.

In the last decade, Benci and Luperi Baglini developed a new theory of generalized functions oriented towards the applications in the field of partial differential equations and of the calculus of variations.
In \cite{ultrafunctions1} and subsequent papers \cite{benci, ultramodel, ultraschwartz, ultraapps}, the authors introduced spaces of ultrafunctions, i.e.\ nonstandard vector spaces of a hyperfinite dimension that extend the space of distributions.
The space of real distributions can be embedded in an algebra of ultrafunctions $V$ such that the following inclusions hold: $\tests(\R)' \subset V \subset \ns{C^1(\R)}$ \cite{ultraschwartz}.
This can be seen as a variation on a result by Robinson and Bernstein, that in \cite{invariant} showed that any Hilbert space $H$ can be embedded in a hyperfinite dimensional subspace of $\ns{H}$.
In the setting of ultrafunctions, some partial differential equations can be formulated coherently by a Galerkin approximation, while the problem of finding the minimum of a functional can be turned to a minimization problem over a formally finite-dimensional vector space.
For a discussion of the applications of ultrafunctions to functional analysis, we refer to \cite{ultrafunctions1, ultramodel, ultraapps}.

Recently, we proposed another algebra of generalized solutions for the study of partial differential equations: the algebra of grid functions $\grid{\Omega}$ defined from an open domain $\Omega \subseteq \R^k$ \cite{ema2}.
This algebra seem particularly suitable for this purpose, mainly due to the following results (proved in \cite{ema2}).
\begin{enumerate}
	\item There exists an embedding from the space of distributions over $\Omega$ to the algebra of grid functions that satisfies the following conditions:
	\begin{itemize}
		\item the pointwise product of grid functions extends the product over $C^0$ functions;
		\item $\D$, the discrete derivative of grid functions, extends the distributional derivative;
		\item the following product rule holds: $\D(u \cdot v)(x) = (\D u)(x)\cdot v(x) + u(x+\varepsilon)\cdot(\D v)(x)$, where $\varepsilon$ is an infinitesimal of a hyperreal field of Robinson's nonstandard analysis. 
	\end{itemize}
	\item It is possible to determine a real vector subspace $\B \subset \grid{\Omega}$ and a surjective homomorphism of vector spaces $\pi : \B \rightarrow \tests(\Omega)'$ which is coherent with the above embedding. 
	\item Each grid function corresponds to a measurable function $\nu: \Omega \rightarrow \mathbb{M}(\R)$, where $\mathbb{M}(\R)$ is the space of positive Radon measures over $\R$.
	This correspondence is also coherent with the homomorphism $\pi$.
\end{enumerate}
Thus the algebra of grid functions provides a generalization both of the space of distributions and the space of Young measures, two spaces of generalized functions customarily used for the study of linear and nonlinear PDEs.
As an initial application of grid functions to the study of ill-posed PDEs, in \cite{illposed} we studied an ill-posed forward-backward parabolic equation. By exploiting the strength of the nonstandard formulation, we were able to characterize the asymptotic behaviour of the solutions and to prove that they satisfy a conjecture formulated by Smarrazzo for the measure-valued solutions to the ill-posed problem \cite{smarrazzo}.

In this paper we will prove that a single grid function
simultaneously represents two different limits (namely, the weak-$\star$ and the Young measure limit) of a sequence of integrable functions. Thus a number of classical concepts (such as different notions of limits and of generalized solutions) can be successfully unified in a relatively elementary but nontrivial hyperfinite setting.
Conversely, grid functions of a finite $L^1$ norm can be described by the weak-$\star$ and the Young measure limit of a sequence of integrable functions. This provides a classic interpretation of grid functions that can be further exploited in the study of PDEs.
As an application of these results, we study a more general version of the problem already discussed in \cite{illposed} and provide a novel definition of solution for such a generalized problem. Moreover, we discuss the interplay between the classical formulation of this problem and the one obtained with grid functions.

We believe that many results relating grid functions and parametrized measures, such as the correspondence between grid functions and parametrized measures or the main theorems of this paper, might be suitably adapted also to spaces of ultrafunctions. So far, however, we are not aware of any research on the connections between these different notions of generalized functions.

\section{Terminology and preliminary notions}\label{prelim}

In this section, we will define the notation and recall some results on grid functions and on Young measures that will be useful in the rest of the paper.

\subsection{Terminology}

We assume that $\Omega \subseteq \R^k$ is an open set.

We will often reference the following real vector spaces:
\begin{itemize}
	\item $\bcf(\R) = \{ f \in C^0(\R) : f \text{ is bounded and } \lim_{|x|\rightarrow \infty} f(x) = 0 \}$.
	\item $C^0_c(\Omega) = \{ f \in C^0_b(\Omega) : \supp f \subset \subset \Omega\}$.
	\item $\tests(\Omega) = \{ f \in C^\infty(\Omega) : \supp f \subset \subset \Omega\}$.
	\item The duality between a vector space $X$ and its dual $X'$ is denoted by $\langle \cdot, \cdot \rangle_X$.
	\item A real distribution over $\Omega$ is an element of $\tests(\Omega)'$, i.e.\ a continuous linear functional $T : \tests(\Omega)\rightarrow\R$.
	If $T$ is a distribution and $\varphi$ is a test function, according to the notation introduced above we denote the action of $T$ over $\varphi$ by
	$ \ldual T, \varphi\rdual$. The distributional derivative is denoted by $D$, so that $DT$ is the distribution defined by $\ldual DT, \varphi \rdual = - \ldual T, \varphi' \rdual$.
	\item $\rad(\R) = \{ \nu : \nu \text{ is a Radon measure over } \R \text{ satisfying } |\nu|(\R)<+\infty \}$.
	\item $\prob(\R) = \{ \nu \in \rad(\R) : \nu \text{ is a probability measure}\}$.
	\item Following \cite{balder, ball, webbym} and others, measurable functions $\nu : \Omega \rightarrow \prob(\R)$ will be called Young measures.
	Measurable functions $\nu : \Omega \rightarrow \rad(\R)$ will be called parametrized measures, even though in the literature the term parametrized measure is used as a synonym for Young measure.
	If $\nu$ is a parametrized measure and if $x \in \Omega$, we will write $\nu_x$ instead of $\nu(x)$.
	\item A Young measure is Dirac if for every $x \in \Omega$ there exists $r \in \R$ such that $\nu_x = \delta_r$. In other words, if a Young measure is Dirac there exists a function $f: \Omega \to \R$ such that $\nu_x = \delta_{f(x)}$. Thus a Dirac Young measure can be identified with a measurable function.
\end{itemize}

\subsection{Grid functions}

Throughout the paper, we will work with a $|\tests(\Omega)'|$-saturated hyperreal field $\hR$, and we will assume familiarity with the basics of Robinson's nonstandard analysis.
For an introduction on the subject, we refer for instance to Goldblatt \cite{go}, but see also \cite{nsa theory apps, da, keisler, nsa working math, nsa robinson}.

For any $x, y \in \ns{\R}$ we will write $x \sim y$ to denote that $x-y$ is infinitesimal,
we will say that $x$ is finite if there exists a standard $M \in \R$ satisfying $|x| < M$, and we will say that $x$ is infinite whenever $x$ is not finite. We will denote by $\fin$ the set of finite numbers in $\hR$, i.e.\ $\fin = \{ x \in \hR : x$ is finite$\}$.

The notion of infinite closeness and of finiteness can be extended componentwise to elements of $\ns{\R^k}$ whenever $k\in\N$.
For any $X \subseteq \hR^k$, $\sh{X}$ will denote the set of the standard parts of the finite elements of $X$.

The set of all hyperreal numbers infinitesimally close to a hyperreal number $x$ is called the monad of $x$ and is denoted by $\mu(x)$.

We will now recall the definition and some properties of grid functions studied in \cite{ema2}.
Grid functions over $\Omega$ are functions defined over a hyperfinite domain that represents $\Omega$.
The hyperfinite domain is obtained as the intersection of $\ns{\Omega}$ with a hyperfinite grid of a uniform step.
We have chosen to work with a uniform grid for a matter of convenience in the representation of the derivative and of the integral (see Definitions \ref{fd} and \ref{def inner}). The relation between finite differences and derivatives have been studied also for some non-uniform grids (see for instance \cite{imme1, imme2}), but we wanted to avoid the complications that arise due to the non-uniform spacing of adjacent elements of the grid.

\begin{definition}[The hyperfinite grid]
	Let $N_0\in\ns{\N}$ be an infinite hypernatural number.
	Set $N = N_0!$ and $\varepsilon = 1/N$, and define
	$$\Lambda = \{ n\varepsilon : n \in  [-N^2, N^2] \cap \ns{\Z}\}.$$
\end{definition}

The choice of working with a hyperfinite grid with endpoints instead of a hypercountable grid $\{n\varepsilon : n \in \Z\}$ allows for a hyperfinite representation of both bounded and unbounded sets. The results presented in the next sections and in other papers on grid functions \cite{ema2,illposed} do not depend upon the choice of endpoints $-N$ and $N$, since we will see that
the grid function representation of distributions and Young measures is uniquely determined by the behaviour of the grid function at the finite points of the grid (for more details see Definition \ref{def equiv} and Theorem \ref{parametrized measures}).

\begin{definition}[The hyperfinite domain $\dom$]
	Define $\dom = \ns{\Omega}\cap \Lambda^k$.
	
	We will say that $x \in \dom$ is nearstandard in $\Omega$ iff there exists $y \in \Omega$ such that $x \sim y$.
\end{definition}

Notice that, since $\dom$ is an internal subset of $\Lambda^k$, it is hyperfinite.

Since no uniform hyperfinite grid includes all the real numbers and since $\Omega$ is open, $\Omega \not \subseteq \dom$. This is in contrast to other hyperfinite representations of uncountable sets that properly include the original standard set.
As already mentioned, a theory of grid functions based on a non-uniform grid that contains all the real numbers would need suitable adjustments at least for the definition of the grid derivative and of the grid integral.

Proposition 2.5 of \cite{ema2} and the hypothesis that $\Omega$ is open ensure that $\sh{\dom} = \overline{\Omega}$. Indeed, this is a consequence of the fact that for every set $\Omega$, not necessarily open, $\sh{\dom}$ is equal to the closure of $\Omega \setminus \{ x \in \Omega \cap \partial \Omega : x \not \in \Q^k\}$ ($\partial \Omega$ denotes the boundary of $\Omega$; we are grateful to an anonymous referee for pointing out this more general result).

Grid functions over $\Omega$ are internal functions over $\dom$.

\begin{definition}[Grid functions over $\Omega$]\label{def grid functions}
	We will say that a grid function over $\Omega$ is an internal function $f : \dom \rightarrow \hR$. The space of grid functions over $\Omega$ is defined as
	$$\grid{\Omega} = \mathbf{Intl}\left(\hR^{\dom}\right) = \{ f : \dom \rightarrow \hR \text{ and } f \text{ is internal}\}.$$
\end{definition}

Since grid functions are defined on a discrete domain, the derivative can be represented by suitable finite difference operators of an infinitesimal step.

\begin{definition}[Some grid derivatives]
	\label{fd}
	For a grid function $f \in \grid{\Omega}$, we define the $i$-th forward finite difference of step $\varepsilon$ as
	\begin{equation*}
	\D^+_i = \D_i f(x)=
	\frac{f(x+\varepsilon e_i)-f(x)}{\varepsilon }.
	\end{equation*}
	If $n \in \ns{\N}$, $\D^n_i$ is recursively defined as $\D_i(\D_i^{n-1})$ and, if $\alpha$ is a multi-index, then $\D^\alpha$ is defined as expected:
	$$
	\D^\alpha f = \D_1^{\alpha_1} \D_2^{\alpha_2} \ldots \D_n^{\alpha_n} f.
	$$
	It is also possible to represent the derivative using backward and centred finite differences of step $\varepsilon$. We will denote the $i$-th backward finite difference by $\D^-_i$. Other finite difference operators that represent the distributional derivative will be discussed in \cite{forthcoming}.
\end{definition}

By using these operators it is possible to define some grid functions counterparts of the gradient and of the divergence.

\begin{definition}[Grid gradient and grid divergence]
	If $f\in\grid{\Omega}$, we define the forward and backward grid gradient of $f$ as
	$\grad^\pm f = (\D^\pm_{1}f, \ldots, \D^\pm_i, \ldots, \D^\pm_{k}f).$
	In a similar way, if $f : \dom \rightarrow \hR^k$, we define the forward and backward grid divergence as
	$\div^\pm f = \sum_{i = 1}^k \D^\pm_{i} f_i.$
\end{definition}

For a discussion of the relevance of the operators $\D^\pm$, $\grad^\pm$ and $\div^{\pm}$ in the theory of grid functions we refer to \cite{ema2}.

In the same spirit, integrals can be replaced by suitable hyperfinite sums.

\begin{definition}[Grid integral and inner product]\label{def inner}
	Let $f,g : \ns{\Omega} \rightarrow \hR$ and let $A \subseteq \dom \subseteq \Lambda^k$ be an internal set.
	We define
	$$
	\int_{A} f(x) d\Lambda^k = \varepsilon^k \cdot \sum_{x \in A} f(x)
	$$
	and
	$$
	\langle f, g \rangle
	=  \displaystyle \int_{\Lambda^k} f(x) g(x) d\Lambda^k
	=  \displaystyle \varepsilon^k \cdot \sum_{x \in \Lambda^k} f(x)g(x),
	$$
	with the convention that, if $x \not \in \ns{\Omega}$, $f(x) = g(x) = 0$.
\end{definition}

For further details about the properties of the grid derivative and the grid integral, we refer to \cite{ema2, imme2, imme1, keisler, nsa working math}.

By using the grid derivative $\D$, it is possible to introduce a grid function counterpart of the space of test functions.

\begin{definition}
	We say that a function $f \in \grid{\Omega}$  is of class $S^0(\Omega)$ iff
	$f(x)$ is finite for some nearstandard $x \in \dom$ and
	for every nearstandard $x, y \in \dom$, $x \sim y$ implies $f(x) \sim f(y)$.
	We say that	$f$ is of class $S^{\infty}(\Omega)$ if $\D^\alpha f \in S^0(\Omega)$ for any standard multi-index $\alpha$.
	
	We define the algebra of grid test functions as follows:
	$$
	\test(\Omega) =
	\left\{ f \in S^{\infty}(\Omega) : \sh{\supp f} \subset \subset \Omega \right\}.
	$$
\end{definition}

In Lemma 3.2 of \cite{ema2} it is proved that the algebra of test function is the grid function counterpart of the space of standard test functions $\tests(\Omega)$ in the following sense:
\begin{itemize}
	\item if $\varphi \in \test(\Omega)$, then $\sh{\varphi} \in \tests(\Omega)$;
	\item if $\varphi \in \tests(\Omega)$, then the restriction of $\ns{\varphi}$ to $\dom$ belongs to $\test(\Omega)$.
\end{itemize}

The duality with grid test functions allows for the definition of a meaningful equivalence relation on the algebra of grid functions.

\begin{definition}\label{def equiv}
	Let $f, g \in \grid{\Omega}$.
	We say that $f \equiv g$ iff
	$
	\langle f, \varphi \rangle \sim \langle g, \varphi \rangle
	$
	for all $\varphi \in \test(\Omega)$.
	We will denote by $[f]$ the equivalence class of $f$ with respect to $\equiv$.
\end{definition}

In \cite{ema2} it is proved that the space of grid functions generalizes the space of distributions.
In particular, there exists a real subspace of $\grid{\Omega}/\equiv$ that is isomorphic to the space of distributions.

\begin{theorem}\label{mainthm}
	Let $\B = \left\{ f\in \grid{\Omega}\ |\ \langle f, \varphi \rangle \text{ is finite for all } \varphi\in\test(\Omega)\right\}.$
	The function $\Phi:(\B/\equiv) \rightarrow \tests(\Omega)'$ defined by
	$$
	\ldual \Phi([f]), \varphi \rdual
	=
	\sh{\langle f, \ns{\varphi} \rangle}
	$$
	is an isomorphism of real vector spaces.
\end{theorem}
\begin{proof}
	See Theorem 3.10 of \cite{ema2}.
\end{proof}

In addition to the above result, in Theorem 3.16 of \cite{ema2} it is also proved that the finite difference operators $\D^+$ and $\D^-$ induces the distributional derivative on the quotient $\B/\equiv$.
Similar results are valid also for the other algebras of generalized functions mentioned in the introduction.

In Theorem 3.19 of \cite{ema2} it is shown how the grid derivative can be used to obtain a non-canonical embedding of the space $\tests(\R)'$ in $\test(\Lambda)'$. The embedding is not canonical since it depends on the choice of a partition of unity and of a Hamel basis for the space of distributions.
Such embeddings are common for algebras of generalized functions: for Colombeau algebras, see e.g. \cite{colombeau 2001}, for asymptotic functions see Section 5 of \cite{oberbuggenberg}, for ultrafunctions see Theorem 1 of \cite{ultraschwartz}.

In the setting of grid functions, sometimes it is more convenient to use other representations than the one provided by the embedding of \cite{ema2}. As an example, we discuss some grid function counterparts of the Dirac distribution.

\begin{example}[Grid functions that represent the Dirac distribution]
	Let $\delta_r$ be the real Dirac distribution centred at some $r \in \R$.
	This distribution is represented by any non-negative grid function $f_r$ satisfying
	\begin{itemize}
		\item $f_r(x) \geq 0$ for every $x \in \Lambda$;
		\item $\supp(f) \subseteq \mu(r)$;
		\item $\sum_{x \in \Lambda} f(x) = 1$.
	\end{itemize}
	However, a more explicit representation is obtained by noticing that, if we define the Heaviside function centred at $r \in \R$ as $H_r : \R\to \R$,
	$$
		H_r(x) = \left\{\begin{array}{ll}0 &\text{if } x \leq r \\ 1 &\text{if } x > r, \end{array}\right.
	$$ 
	then $\delta_r = D H_r$.
	From this representation and by using the grid derivative instead of the distributional derivative we obtain that a grid function representative of the Dirac distribution $\delta_r$ is $d_r = \D(\ns{{H_r}|_{\Lambda}})$. This grid function can be written explicitly by introducing the number $r^- = \max_{x \in \Lambda}\{x \leq r \}$:
	$$
		d_r(x) = \left\{\begin{array}{ll}0 &\text{if } x \ne r^- \\ \varepsilon^{-1} &\text{if } x = r^-. \end{array}\right.
	$$
	Similar representations can be obtained by using the backward or centred finite differences instead of the forward finite difference $\D$.
\end{example}

\subsection{Grid functions of a finite $L^1$ norm}

In the sequel, we will use the following $L^p$ norms over the space of grid functions.

\begin{definition}[$L^p$ norms for grid functions]
	For all $f \in \grid{\Omega}$, define
	$$
	\norm{f}_p^p = \varepsilon^k \sum_{x \in \dom} |f(x)|^p \text{ if } 1 \leq p < \infty, \text{ and } \norm{f}_\infty = \max_{x \in \dom} |f(x)|.
	$$
\end{definition}

In Lemma 4.1 of \cite{ema2} it is proved that if $\norm{f}_p \in \fin$ for some $p$, then $f \in \B$, i.e. $[f]$ is a well-defined distribution.
In this paper we will use also the following property: if $\norm{f}_1 \in \fin$, then $[f]$ is an element of the dual of $C^0_c(\Omega)$.

\begin{proposition}\label{proposizione nuova l1fin}
	If $f \in \grid{\Omega}$ satisfies $\norm{f}_1 \in \fin$, then $[f] \in C^0_c(\Omega)'$, i.e.\ it can be identified with a continuous linear functional over $C^0_c(\Omega)$, that we will still denote by $[f]$, defined by
	$$
		\langle [f], \varphi \rangle_{C^0_c(\Omega)} = \sh{\langle f, \ns{\varphi}\rangle}
	$$
	for every $\varphi \in C^0_c(\Omega)$.
\end{proposition}
\begin{proof}
	Let $\varphi \in C^0_c(\Omega)$: then $\varphi \in L^\infty(\Omega)$, so that $\norm{\varphi}_\infty\sim\norm{\ns{\varphi}}_\infty \in \fin$.
	By the discrete H\"older's inequality,
	$$
		|\langle f, \ns{\varphi}\rangle| \leq \norm{f}_1 \norm{\ns{\varphi}}_\infty \in \fin.
	$$
	This estimate and linearity of the hyperfinite sum over $\dom$ allow to conclude that $[f]$ is a linear functional over $C^0_c(\Omega)$.
	
	In order to prove continuity it is sufficient to notice that if $\varphi, \psi \in S^0(\Omega)$ satisfy $\norm{\varphi-\psi}_\infty \sim 0$, then
	$$
		|\langle f, \varphi-\psi\rangle| \leq \norm{f}_1 \norm{\varphi-\psi}_\infty \sim 0.
	$$
	As a consequence, $[f]$ is a continuous linear functional over $C^0_c(\Omega)$, as desired.
\end{proof}

\subsection{Young measures}

We find it useful to recall some definitions and results on Young measures.

\begin{definition}
	Let $\nu : \Omega \rightarrow\prob(\R)$ be a Young measure.
	If $g\in C^0_b(\R)$, the composition $g(\nu)$ is defined by
	$$
	g(\nu(x)) = \int_{\R} g d \nu_x.
	$$
\end{definition}

It is well-known that Young measures are able to express the weak-$\star$ limit in $L^\infty$ of the composition between a bounded sequence of $L^1$ functions with a function in $C^0_b(\R)$.
This result is a consequence of the fundamental theorem of Young measures.

\begin{theorem}\label{thm young measures}
	For every bounded sequence of $L^1(\Omega)$ functions $\{z_n\}_{n \in \N}$, there exists a subsequence $\{z_{n_k}\}_{k \in \N}$ of $\{z_n\}_{n \in \N}$ and a Young measure $\nu$ such that for all $g\in C^0_b(\R)$ and for all $\varphi \in C^0_c(\Omega)$,
	\begin{eqnarray*}
			\displaystyle \lim_{k \rightarrow \infty} \int_{\Omega} g(z_{n_k}(x)) \varphi(x) dx
			& = &
			\displaystyle \int_{\Omega} \left(\int_{\R} g d \nu_x \right) \varphi(x) dx\\
			& = &
			\displaystyle \int_{\Omega} g(\nu_x) \varphi(x) dx.
	\end{eqnarray*}
	In other words, $g(z_n) \weaklys g(\nu)$ in $L^\infty(\Omega)$ for all $g\in C^0_b(\R)$.
\end{theorem}
\begin{proof}
	See e.g. \cite{balder,ball,bonnetier} and references therein.
\end{proof}

In the last statement of Theorem \ref{thm young measures},
we have used density of $C^0_c(\Omega)$ in $L^1(\Omega)$.

\begin{definition}
	If  $\{z_n\}_{n \in \N}$ is a bounded sequence of $L^1(\Omega)$ functions and if $\nu$ is a Young measure that satisfies Theorem \ref{thm young measures}, we will say that $\{z_n\}_{n \in \N}$ converges to $\nu$ in the sense of Young measure and we will write $z_n \young \nu$.
\end{definition}

The relations between grid functions and parametrized measures (including Young measures), are studied in depth in \cite{ema2}.
We recall the main results that will be useful for this paper.

\begin{theorem}\label{parametrized measures}
	For every $f \in \grid{\Omega}$, there exists a parametrized measure $\nu^f : \Omega \rightarrow \rad(\R)$ such that for all $g \in \bcf(\R)$ and for all $\varphi \in C^0_c(\Omega)$
	\begin{equation}\label{young equivalence equation}
	\sh{\langle \ns{g}(f), \ns{\varphi} \rangle}
	=
	\int_{\Omega} \left( \int_{\R} g d \nu^f_x \right) \varphi(x) dx.
	\end{equation}
	Moreover,
	\begin{enumerate}
		\item for every Young measure $\nu$ over $\Omega$ there exists a grid function $f$ such that $\nu^f = \nu$;
		\item for all $x \in \Omega$ and for all Borel $A \subseteq \R$, $0\leq \nu^f_x(A) \leq 1$;
		\item\label{2} if $\norm{f}_p \in \fin$ for some $1 \leq p \leq \infty$, then $\nu^f$ is a Young measure.
	\end{enumerate} 
\end{theorem}
\begin{proof}
	For the proof of point (1), see Theorem 2.9 of \cite{cutland}.
	The other statements are proved in Theorem 4.12, Theorem 4.14 and Proposition 4.17 of \cite{ema2}.
\end{proof}

The difference between $\nu^f_x(\R)$ and $1$ is due to $f$ assuming infinite values in some non-negligible fraction of $\mu(x)\cap\Lambda^k$.
Point \eqref{2} of Theorem \ref{thm young measures} can be rephrased in the following way: if $\norm{f}_p \in \fin$ for some $1 \leq p \leq \infty$, then $f$ assumes infinite values only on a (possibly empty) set $\Omega_{in\!f} \subseteq \dom$ of Loeb measure $0$.

\begin{corollary}\label{cor parametrized measures}
	 For every $f \in \grid{\Omega}$ and for every $g \in \bcf(\R)$, $\int_{\R} g(\tau) d\nu^f = \int_{\R} \tau d\nu^{\ns{g(f)}}$.
\end{corollary}
\begin{proof}
	By Theorem \ref{parametrized measures},
	$$
	\sh{\langle \ns{g(f)}, \ns{\varphi}\rangle}
	=
	\int_{\Omega} \left( \int_{\R} g d \nu^f_x \right) \varphi(x) dx
	$$
	and, since the hypothesis $g \in \bcf(\R)$ entails $\norm{\ns{g(f)}}_\infty \in \fin$, we have also
	$$
	\sh{\langle \ns{g(f)}, \ns{\varphi}\rangle}
	=
	\int_{\Omega} \left( \int_{\R} \tau d \nu^{\ns{g(f)}}_x \right) \varphi(x) dx.
	$$
\end{proof} 

\begin{lemma}\label{corollario baricentro}
	For every $f \in \grid{\Omega}$, let $\nu^f : \Omega \rightarrow \rad(\R)$ the parametrized measure satisfying Theorem \ref{parametrized measures}, and let $f_b : \Omega \rightarrow \R$ be its barycentre, defined by
	$$
	f_b(x) = \int_{\R} \tau d\nu^f_x.
	$$
	Then $f_b$ is a measurable function.
	Moreover, if $\norm{f}_1 \in \fin$, then $f_b \in L^1(\Omega)$ and $\norm{f_b}_1\leq\norm{f}_1$.
\end{lemma}
\begin{proof}
See Corollary 3.15 of \cite{illposed}.
\end{proof}

\section{The main results}

We are now ready to prove that grid functions are expressive enough to describe simultaneously both the weak-$\star$ limit and the Young measure limit of bounded sequences of integrable functions.

\begin{theorem}\label{thm bounded -> grid}
	For every bounded sequence $\{z_n\}_{n \in \mathbb{N}}$ in $L^1(\Omega)$ such that
	\begin{itemize}
		\item $z_{n} \weaklys z_\infty$ in  $C^0_c(\Omega)'$ and
		\item $z_n \young \nu$,
	\end{itemize}
	there exists a (non unique) function $z \in \mathbb{G}(\Omega)$ such that
	\begin{enumerate}
		\item for all $\varphi \in C^0_c(\Omega)$
		$$
		\langle z, \ns{\varphi}\rangle
		\sim
		\langle z_{\infty}(x), \varphi(x) \rangle_{C^0_c(\Omega)};
		$$
		\item for all $g \in C^0_b(\mathbb{R})$	and for all $\varphi \in C^0_c(\Omega)$
		$$
		\langle \ns{g}(z), \ns{\varphi}\rangle
		\sim
		\int_{\Omega} g(\nu(x)) \varphi(x) dx.
		$$
	\end{enumerate}
\end{theorem}
\begin{proof}
	Since $\tests(\Omega) \subseteq C^0_c(\Omega)$, recall that $z_\infty \in C^0_c(\Omega)' \subseteq \tests(\Omega)'$ can be identified with a distribution (that we will still denote by $z_\infty$) by posing
	$$
		\langle z_\infty, \varphi \rangle_{\tests(\Omega)}
		=
		\langle z_\infty, \varphi \rangle_{C^0_c(\Omega)}
	$$
	for every $\varphi \in \tests(\Omega)$.
	
	Let also $b: \Omega \rightarrow \R$ be the barycentre of $\nu$: $b(x) = \int_\R \tau d \nu_x$.
	The hypotheses over $\nu$ are sufficient to entail $b \in L^1(\Omega)$ (see e.g.\ Corollary 3.13 of \cite{webbym}).
	Thus the function $b$ can be identified with a distribution (that we will still denote by $b$) by posing
	$$
	\langle b, \varphi \rangle_{\tests(\Omega)} = \int_{\Omega} b(x) \varphi(x) dx
	$$
	for every $\varphi \in \tests(\Omega)$.

	Thanks to Theorem \ref{mainthm}, there exists a grid function $z^D\in \grid{\Omega}$ that corresponds to the distribution $z_\infty$.
	The grid function $z^D$ might not correspond to the Young measure $\nu$; however $\nu^{z^D}$ and $\nu$ have the same barycentre $b$.
	
	To see that this is the case, consider the grid functions
	$$z_n^D(x) = \left\{
	\begin{array}{ll}
	z^D(x) & \text{if } |z^D(x)|\leq n;\\
	0  & \text{if } |z^D(x)|> n.
	\end{array}
	\right.
	$$
	defined for every $n\in\N$, and let $b_n$ be the barycentre of $\nu^{z_n^D}$: $b_n(x) = \int_{\R} \tau d\nu^{z_n^D}$.
	By this definition it is easy to see that $b_n(x) = b(x)$ for every $x \in \Omega$ such that $|z^D(y)| \leq n$ for every $y \in \dom$, $y \sim x$.
	We have already observed that the hypothesis $\norm{z^D}_1 \in \fin$ ensures that the set $\Omega_{in\!f}=\{x : z^D(x) \text{ is infinite}\}$ has Loeb measure $0$.
	As a consequence, $\lim_{n \rightarrow \infty} b_n(x) = b(x)$ for a.e.\ $x \in \Omega$.
	
	Now let $g_n \in C^0_b(\R)$ with $g(\tau) = \tau$ for every $\tau \in [-n,n]$. By Theorem \ref{parametrized measures}, we have that for all $\varphi \in C^0_c(\Omega)$
	\begin{eqnarray}
		\notag\label{eqn iniziale per baricentro}
		\langle \ns{g}(z^D), \ns{\varphi}\rangle 
		&\sim&
		\int_{\Omega} \left( \int_{\R} g d \nu^{z^D}_x \right) \varphi(x) dx\\
		\notag
		&=&
		\int_{\Omega} \left(\int_{[-n,n]} \tau d\nu^{z^D}_{x}\right) \varphi(x) dx + \int_{\Omega} \left( \int_{\{x \in \R : |x| > n\}} g d \nu^{z^D}_x \right) \varphi(x) dx\\
		\notag\label{eqn per baricentro}
		&=&
		\int_{\Omega} b_n(x) \varphi(x) dx + \int_{\Omega} \left( \int_{\{x \in \R : |x| > n\}} g d \nu^{z^D}_x \right) \varphi(x) dx.
	\end{eqnarray}
	From the previous equalities we obtain
	\begin{equation}\label{eqn 3 per baricentro}
		\left|
		\int_{\Omega} \left( \int_{\R} g d \nu^{z^D}_x \right) \varphi(x) dx
		-
		\int_{\Omega} b_n(x) \varphi(x) dx
		\right|
		=
		\left|
		\int_{\Omega} \left( \int_{\{x \in \R : |x| > n\}} g d \nu^{z^D}_x \right) \varphi(x) dx.
		\right|
	\end{equation}
	Since $g \in C^0_b(\R)$ entails that $g$ is bounded, $\sup_{|x|>n} |g(x)$ is well-defined.
	Thus
	$$
		\left| \int_{\{x \in \R : |x| > n\}} g d \nu^{z^D}_x \right| \leq \sup_{|x|>n} |g(x)|
	$$
	and
	$$
	\left| \int_{\Omega} \left( \int_{\{x \in \R : |x| > n\}} g d \nu^{z^D}_x \right) \varphi(x) dx \right| \leq \sup_{|x|>n} |g(x)| \norm{\varphi}_1.
	$$
	From the last estimate and from equation \eqref{eqn 3 per baricentro} we obtain
	\begin{equation}\label{eqn 4 per baricentro}
	\left|
	\int_{\Omega} \left( \int_{\R} g d \nu^{z^D}_x \right) \varphi(x) dx
	-
	\int_{\Omega} b_n(x) \varphi(x) dx
	\right|
	\leq \norm{\varphi}_1 \sup_{|x|>n} |g(x)|.
	\end{equation}

	Since $\lim_{|x|\rightarrow \infty} g(x) = 0$, $\lim_{n \rightarrow \infty} \left(\sup_{|x|>n} |g(x)|\right) = 0$.
	As a consequence, taking the limit as $n\rightarrow \infty$ in equation \eqref{eqn 4 per baricentro} and taking into account the arbitrariness of $\varphi$, we obtain that the barycentre of $\nu^{z^D}$ is $\lim_{n \rightarrow \infty} b_n = b$, as desired.
	
	By Theorem \ref{parametrized measures}, there exists a grid function $z_0\in \grid{\Omega}$ that corresponds to the Young measure $\nu-\nu^{z^D}$. By the previous part of the proof, this Young measure has null barycentre, i.e.\
	$$
		\int_{\R} \tau d\nu^{z_0}_x =
		\int_{\R} \tau d(\nu - \nu^{z^D})_x =
		0
	$$
	for every $x \in \R$.
	
	We claim that the grid function $z=z^D+z_0$ satisfies the desired properties.
	In fact, for all $\varphi \in \tests(\Omega)$
	\begin{eqnarray*}	
	\langle z, \ns{\varphi}\rangle
	&=&
	\langle z^D, \ns{\varphi}\rangle + \langle z_0, \ns{\varphi}\rangle\\
	&\sim&
	\langle z_\infty, \varphi\rangle_{\tests(\Omega)}
	+
	\int_{\Omega} \int_{\R} \tau d\nu^{z_0}_x \varphi dx\\
	&=&
	\langle z_\infty, \varphi\rangle_{\tests(\Omega)}
	+
	\int_{\Omega} \int_{\R} \tau d(\nu - \nu^{z^D})_x \varphi dx\\
	&=&
	\langle z_\infty, \varphi\rangle_{\tests(\Omega)}.
	\end{eqnarray*}
	Then, since $\tests(\Omega)$ is dense in $C^0_c(\Omega)$, we conclude that also
	$
	\langle z, \ns{\varphi}\rangle
	\sim
	\langle z_\infty, \varphi\rangle_{C^0_c(\Omega)}
	$
	for every $\varphi \in C^0_c(\Omega)$.
	 
	Similarly, for all $\varphi \in C^0_c(\Omega)$,
	\begin{eqnarray*}
	\langle \ns{g}(z), \ns{\varphi}\rangle
	&\sim&
	\int_{\Omega} \left( \int_{\R} g d(\nu^{z^D}+\nu - \nu^{z^D})_x \right) \varphi(x) dx\\
	&=&
	\int_{\Omega} \left( \int_{\R} g d\nu_x \right) \varphi(x) dx,
	\end{eqnarray*}
	as desired.
\end{proof}

The possibility of representing simultaneously these two limits of a bounded sequence of integrable functions is particularly relevant when the sequence features both concentrations and oscillations.
Classically, the behaviour of such sequences can only by described by the combination of the weak-$\star$ limit, describes concentrations but not oscillations, and the Young measure limit, that describes oscillations but not concentrations.
Instead, we are able to express both behaviours with a unique grid function.

\begin{example}
	Let $z: \R \to \R$ be the function of period $2$ satisfying
	$$
		z(x) = \left\{
		\begin{array}{ll}
		-1 & \text{if } -1 \leq x < 0\\
		1 & \text{if } 0 \leq x < 1.
		\end{array}
		\right.	
	$$
	Let $z_n : (-1,1) \to \R$ defined by $z_n(x) = z(nx) + n\chi_{[-n^{-1},n^{-1}]}$.
	Notice that $\norm{z_n}_1 = 4$ for all $n \in\N$, $z_n \weaklys 2\delta_0$ and $z_n \young \frac{1}{2} \left(\delta_{-1} + \delta_{1}\right)$. As already observed, the concentration is described only by the weak-$\star$ limit and the oscillation only by the Young measure limit.
	
	A grid function representative of this sequence is
	$$
		f(n\varepsilon) = \left\{
		\begin{array}{ll}
		N -1 & \text{if } n = -1\\
		N+1 & \text{if } n = 0\\
		-1^n & \text{otherwise.}
		\end{array}
		\right.
	$$
	To see that this is the case, consider at first $\nu^f$, the Young measure corresponding to $f$. By Corollary 4.15 of \cite{ema2}, this is equal to $\nu^{\tilde{f}}$, the Young measure corresponding $\tilde{f}(n\varepsilon) = -1^n$, since the set $\{x \in \Lambda : f(x) \ne \tilde{f}(x) \}$ has null Loeb measure.
	By Proposition 4.17 of \cite{ema2}, $\nu^{\tilde{f}} = \frac{1}{2} \left(\delta_{-1} + \delta_{1}\right)$.
	Notice that this Young measure is constant and its barycentre $f_b$ is null.
	
	In order to determine the distribution corresponding to $f$ we can evaluate the product
	$\langle f, \ns{\varphi} \rangle$ for every $\varphi \in C^0(\R)$.
	We have
	\begin{align*}
		\langle f, \ns{\varphi} \rangle & = \frac{1}{\varepsilon}\left( N\ns{\varphi}(-\varepsilon) + N \ns{\varphi}(0) + \sum_{n = -N}^{N} -1^n \varphi(n\varepsilon)\right)\\
		& = \ns{\varphi}(-\varepsilon) + \ns{\varphi}(0) + \frac{1}{\varepsilon}\sum_{n = -N}^{N} -1^n \varphi(n\varepsilon).
	\end{align*}
	By Theorem \ref{parametrized measures},
	$$
		\frac{1}{\varepsilon}\sum_{n = -N}^{N} -1^n \varphi(n\varepsilon) \approx \int_{\R} f_b(x) \varphi(x) dx = 0.
	$$
	Thus, taking into account that $\varphi \in C^0(\R)$ entails $\ns{\varphi(-\varepsilon)} \approx \ns{\varphi(0)} = \varphi(0)$,
	$$
		\langle f, \ns{\varphi} \rangle = \ns{\varphi}(-\varepsilon) + \ns{\varphi}(0) \approx 2\varphi(0),
	$$
	i.e. $[f] = 2 \delta_0$.
	In conclusion, $\{z_n\}_{n\in\N} \weaklys [f]$ and $\{z_n\}_{n\in\N} \young \nu^f$, as claimed.
\end{example}

The following result is an immediate consequence of Theorem \ref{mainthm} and of Theorem \ref{thm bounded -> grid}.

\begin{theorem}\label{thm isomorfismo L1}
	Let $\mathbb{L}^1(\Omega) = \{f \in \grid{\Omega} : \norm{f}_1 \in \fin \}$.
	The function $\Psi: \mathbb{L}^1(\Omega)/\!\!\equiv\,\, \rightarrow C^0_c(\Omega)'$ defined by
	$$
	\ldual \Psi([f]), \varphi \rangle_{C^0_c(\Omega)}
	=
	\sh{\langle f, \ns{\varphi} \rangle}
	$$
	is an isomorphism of real vector spaces.
\end{theorem}
\begin{proof}
	The function $\Psi$ is well-posed: let $f, g \in \mathbb{L}^1(\Omega)$ satisfy $f \equiv g$.
	Then
	$$
	\sh{\langle f, \ns{\varphi} \rangle}
	=
	\sh{\langle g, \ns{\varphi} \rangle}
	$$
	for every $\varphi \in \tests(\Omega)$.
	Since $\tests(\Omega)$ is dense in $C^0_c(\Omega)$, we deduce that $\Psi([f]) =  \Psi([g])$ also in $C^0_c(\Omega)'$.
	
	Similarly, injectivity of $\Psi$ is a consequence of the injectivity of $\Phi$ (see Theorem \ref{mainthm}) and of density of $\tests(\Omega)$ in $C^0_c(\Omega)$.
	
	Finally, surjectivity of $\Psi$ is a consequence of Theorem \ref{thm bounded -> grid} and of the fact that $L^1(\Omega)$ is dense in $C^0_c(\Omega)'$ with respect to the weak-$\star$ topology, so that every $\mu \in C^0_c(\Omega)'$ can be obtained as the weak-$\star$ limit of a sequence of functions in $L^1(\Omega)$.
\end{proof}

We will now prove the converse of Theorem \ref{thm bounded -> grid}, namely that every grid function with a finite $L^1$ norm corresponds simultaneously to the weak-$\star$ limit and the Young measure limit of a sequence of integrable functions.

\begin{theorem}\label{thm grid -> bounded}
	For every grid function $z \in \grid{\Omega}$, if $\norm{z}_1 \in \fin$, there exists a sequence $\{z_n\}_{n \in \mathbb{N}}$ in $L^1(\Omega)$ such that
	\begin{enumerate}
		\item\label{item weaklys} $z_n \weaklys [z]$ in $C^0_c(\Omega)'$,
		\item\label{item young} $z_n \young \nu^z$.
	\end{enumerate}
\end{theorem}
\begin{proof}
	The proof is based upon the following results on distributions and Young measures, respectively.
	\begin{enumerate}
		\item[(d)]
		For every distribution $T$ there is a sequence $\{d_n\}_{n\in\N}$ in $C^\infty_c(\Omega)$ such that for every $\varphi \in \tests(\Omega)$
		$$
		\lim_{n \rightarrow \infty} \int_{\Omega} d_n \varphi dx
		=
		\ldual T, \varphi \rdual.
		$$
		See e.g. Section 6.6 of \cite{strichartz}.
		Notice that, since $C^\infty_c(\Omega) \subseteq L^1(\Omega)$, $\{d_n\}_{n\in\N}$ is also weakly-$\star$ convergent in $C^0_c(\Omega)'$ to a continuous linear functional, still denoted by $T$, defined by
		$$
		\ldual T, \varphi \rangle_{C^0_c(\Omega)}
		=
		\lim_{n \rightarrow \infty} \int_{\Omega} d_n \varphi dx.
		$$
		\item[(Y)]
		For every Young measure $\nu$ there is a sequence $\{y_n\}_{n\in\N}$ in $L^1(\Omega)$ such that for every $g \in \bcf$ and for all $\varphi \in \tests(\Omega)$
		$$
		\lim_{n \rightarrow \infty} \int_{\Omega} g(y_n) \varphi dx
		=
		\int_{\Omega} \left(\int_{\R}g d\nu_x\right) \varphi(x) dx.
		$$
		See e.g. Theorem 1.1 of \cite{bonnetier} and references therein.
	\end{enumerate}
	
	Let $\{d_n\}_{n\in\N}$ be a sequence in $\tests(\Omega)$ satisfying condition (d) with $T = [z]$.
	Denote by $\nu^d$ the Young measure limit of $\{d_n\}_{n\in\N}$, i.e.\ $d_n \young \nu^d$. Notice that it is not necessary that $\nu^d = \nu^z$.
	However, an argument similar to that of the proof of Theorem \ref{thm bounded -> grid} allows to conclude that the Young measures $\nu^z$ and $\nu^d$ have the same barycentre. For the purposes of this proof, it is more convenient to rephrase this result by saying that the barycentre of $\nu^z - \nu^d$ is null.

	Let $\{y_n\}_{n\in\N}$ be a sequence in $\tests(\Omega)$ satisfying condition (Y) with $\nu = \nu^z - \nu^d$.
	Finally, define $z_n = d_n+y_n$ for every $n \in \N$. We claim that $\{z_n\}_{n\in\N}$ satisfies the desired conditions.

	\eqref{item weaklys} Let $\varphi \in \tests(\Omega)$: by defintion of  $\{z_n\}_{n\in\N}$, by linearity of the limit and by recalling that the Young measure limit of  $\{y_n\}_{n\in\N}$ has a null barycentre,
	\begin{eqnarray*}
		\lim_{n \rightarrow \infty} \ldual z_n, \varphi \rdual
		& = &
		\lim_{n \rightarrow \infty} \ldual b_n, \varphi \rdual + \lim_{n \rightarrow \infty} \ldual y_n, \varphi \rdual\\
		&=&
		\ldual [z], \varphi \rdual + 0\\
		&\sim &
		\langle z, \ns{\varphi} \rangle.
	\end{eqnarray*}
	This equality and density of $\tests(\Omega)$ in $C^0_c(\Omega)$ allow to conclude that $\{z_n\}_{n\in\N}$ is also weakly-$\star$ convergent in $C^0_c(\Omega)'$ to $[z]$.
	
	\eqref{item young}
	By definition, $z_n \young \nu^d+\nu^z-\nu^d = \nu^z$, as desired.
\end{proof}

The proof of Theorem \ref{thm grid -> bounded} provides an interpretation of the infinite and finite part of a grid function $f \in \grid{\Omega}$ with $\norm{f}_1 \in \fin$:
\begin{itemize}
	\item the finite part of $f$ corresponds to a Young measure $\nu^f$ over $\Omega$; the barycentre $f_b$ of $\nu^f$ belongs to $L^1(\Omega)$;
	\item the infinite part of $f$, that corresponds to the distribution $[f]-f_b$, is a grid function representative of a Radon measure whose support is a null subset of $\Omega$.
\end{itemize}

Theorem \ref{thm grid -> bounded} gives also a classical interpretation of grid functions of a finite $L^1$ norm.
Its importance can be appreciated by taking into account that, in order to better understand some results obtained with nonstandard techniques, it is useful to have functions that act as bridges between classical and nonstandard mathematics. In one direction, the operator $\ns $ is sufficient to turn classical object into standard ones. In the opposite direction, the most common of such bridges is the notion of standard part of a number, that can be extended also to functions. However, this extension is well-defined only for continuous functions. Thus it is fundamental to devise other relevant extensions of the standard part that can be applied also when the target set is a space of generalized functions. In the setting of grid functions, these extensions are provided by results such as Theorem \ref{mainthm}, Theorem \ref{parametrized measures} and the new Theorems \ref{thm isomorfismo L1} and \ref{thm grid -> bounded}.

The usefulness of these novel results will become more evident in Section \ref{sec 4}, where we will show how a grid formulation of a class of ill-posed PDEs can be used to define a classical measure-valued solution that coherently extends other notions of solution already introduced for particular instances of this problem. Moreover, as we will discuss explicitly in Section \ref{sec applicazione thm}, Theorem \ref{thm grid -> bounded} also suggests that the grid function formulation of a PDE corresponds to a suitable family of classical regularized problems.

\section{An application of Theorems \ref{thm bounded -> grid} and \ref{thm grid -> bounded} to the study of a class of ill-posed nonlinear PDEs}\label{sec 4}

As an application of Theorems \ref{thm bounded -> grid} and \ref{thm grid -> bounded}, i.e.\ of the correspondence between grid functions and the two measure-valued limits of integrable functions, consider the Neumann initial value problem
\begin{equation}\label{mark}
	\left\{
	\begin{array}{l}
	\partial_t u=\Delta \phi(u) \text{ in } \Omega\\
	\frac{\partial \phi(u)}{\partial \hat{n}} = 0 \text{ in } [0,T] \times \partial\Omega\\
	u(0,x) = u_0(x)
	\end{array}\right.
\end{equation}
with a non-monotone $\phi: \R \rightarrow \R$ and on a domain $\Omega$ that is open, bounded and with a smooth boundary $\partial \Omega$.
The hypothesis that $\phi$ is non-monotone entails that problem \eqref{mark} is ill-posed forward in time in the intervals where $\phi$ is decreasing. Consequently, problem \eqref{mark} only has measure-valued solutions. For a comprehensive discussion of problem \eqref{mark} and of its relevance for some applications we refer to \cite{illposed,plotnikov,smarrazzo}.

In \cite{illposed}, we have provided the following grid function formulation for problem \eqref{mark}. Begin by defining
$I_x^+ = \{ i : x + \varepsilon e_i \not \in \dom \}$ and $I_x^- = \{ i : x - \varepsilon e_i \not \in \dom \}$.
Then, for $u \in \grid{\dom}$ let
\begin{eqnarray*}
\lap \ns{\phi(u(t,x))}
&=&
-\varepsilon^{-1}\sum_{i \in I_x^+} \D^-_i \ns{\phi(u(t,x))}
+
\varepsilon^{-1}\sum_{i \in I_x^-} \D^+_i \ns{\phi(u(t,x))}+\\
&&+
\sum_{i \not \in I_x^+ \cup I_x^-}\D_i^+ \D_i^- \ns{\phi(u(t,x))}.
\end{eqnarray*}
As argued in Section 4 of \cite{illposed}, this is a first-order discrete approximation of the Laplacian with Neumann boundary conditions.
The corresponding grid function formulation of problem \eqref{mark} is
\begin{equation}\label{pb discreto}
	\left\{ \begin{array}{l}
	u_t =\Delta_\Lambda\ns{\phi}(u)\\
	u\left(0,x\right)= P(u_0)(x),
	\end{array}\right.
\end{equation}
where $P(u_0)$ is the $L^2$ projection of $\ns{u_0}$ to the closed subspace $\grid{\dom}$ (see Definition 4.4 of \cite{ema2} or Definition 3.11 of \cite{illposed}).

Problem \ref{mark} is usually studied assuming that
\begin{enumerate}
	\item $\phi \in C^1(\R)$;
	\item $\phi(x) \geq 0$ for all $x \geq 0$ and $\phi(0) = 0$;
	\item there exists $u^-, u^+  \in \R$ with $0 < u^- < u^+$ such that $\phi'(u) > 0$ if $u \in (0,u^-) \cup (u^+,+\infty)$ and $\phi'(u) < 0$ for $u \in (u^-, u^+)$, or
	\item there exists $u^- \in \R$ with $0 < u^-$ such that $\phi'(u) > 0$ if $u \in (0,u^-)$ and $\phi'(u) < 0$ for $u \in (u^-,+\infty)$ and $\lim_{x \rightarrow +\infty} \phi(x) =0$;
	\item $u_0 \in L^\infty(\Omega)$ and $u_0(x) \geq 0$ for all $x\in\Omega$.
\end{enumerate}
Under these hypotheses we have shown that the solution to the grid function formulation corresponds to the sum of the weak-$\star$ limit and the Young measure limit of a sequence of $L^1$ solutions of a regularized problem.
For a more precise statement, we refer to Theorem 5.7 of \cite{illposed}.

To the best of our knowledge, problem \eqref{mark} has only been studied under the hypotheses (1)--(5) above; however, the grid function formulation \eqref{pb discreto} has a unique global solution with good physical properties even if one drops assumptions (3) or (4) and replaces (1) and (5) with the weaker
\begin{enumerate}
	\item[(1')] $\phi$ is Lipschitz continuous;
    \item[(5')] $u(0,x) \geq 0$ for all $x \in \dom$, $u(0,\cdot)\in\mathbb{L}^1(\Omega)$ and,
    if $\phi \not \in L^\infty(\R)$, $\norm{u(0,\cdot)}_\infty \in\fin$.
\end{enumerate}
In light of Theorems \ref{parametrized measures}, \ref{thm bounded -> grid} and \ref{thm grid -> bounded}, the weaker hypothesis (5') allows for the representation of measure-valued initial data obtained from sequences of integrable functions. As already argued, these measure-valued initial data correspond to the sum of a Young measure and of a non-negative Radon measure.

Despite these weaker hypotheses, many results obtained in \cite{illposed} are still valid.

\begin{proposition}\label{prop prop soluzione pb discreto}
	Consider problem \eqref{pb discreto} under the hypoteses (1'), (2) and (5') above.
	\begin{itemize}
		\item Problem \eqref{pb discreto} has a unique global solution $u \in \ns{C^1(\ns{[0,+\infty)}, \grid{\dom})}$.
		\item $\norm{u(t,\cdot)}_1 = \norm{u(0,\cdot)}_1$ for all $t \geq 0$.
		\item For any $g \in C^1(\R)$ with $g' \geq 0$, define
		$
		G(u(t,x)) = \int_0^{u(t,x)} g(\phi(s)) ds.
		$
		Then, $u$ satisfies the entropy condition
		\begin{equation}\label{eq entropy}
			\ns{G}(u)_t = \div^-((\ns{g}(\phi(u)) \grad^+(\phi(u))) - \grad^- \ns{g}(\phi(u)) \cdot \grad^- \phi(u).
		\end{equation}
		\item $[u]\in \tests'(\R\times\Omega)$, $[\ns{\phi}(u)] \in L^\infty(\R\times\Omega)$, and $[u]$ and $[\ns{\phi}(u)]$ satisfy
		\begin{equation}\label{less regular}
		\int_0^T \langle [u], \varphi_t \rangle + \langle [\ns{\phi}(u)], \Delta \varphi \rangle dt + \ldual [u](0,x) \varphi(0,x) \rangle_{C^0(\Omega)} = 0
		\end{equation}
	\item for almost every initial data $u(0,\cdot)$, $u(t)$ converges to a steady state $\tilde{u}$ satisfying $\ns{\phi(\tilde{u})} = 0$ for every $x \in \dom$ and $\ns{\phi'(\tilde{u}(x))}<0$ for at most one $x \in \dom$.
	\end{itemize}
\end{proposition}
\begin{proof}
	These assertions can be obtained from the corresponding results in \cite{illposed}, whose proofs do not depend on hypotheses (3), (4) and (5) over $\phi$.
	
	In order to prove \eqref{less regular} we need to prove that $\norm{\ns{\phi(u)}}_\infty \in \fin$ regardless of the behaviour of $\phi$.
	To see that this is the case, if $\norm{\phi}_\infty < +\infty$ then the desired result is trivially true, since $\norm{\ns{\phi(u)}}_\infty \leq \norm{\ns{\phi}}_\infty \in \fin$. If $\norm{\phi}_\infty = +\infty$, hypothesis (2) entails $\lim_{x \rightarrow +\infty} \phi(x) = +\infty$. Then the hypothesis $\norm{u(0,\cdot)}_\infty \in \fin$ and an argument similar to that of point 2.\ of Proposition 4.6 of \cite{illposed} entails $\norm{u}_\infty \in \fin$, so that also $\norm{\ns{\phi(u)}}_\infty \in \fin$.
\end{proof}

The entropy condition \eqref{eq entropy} is the grid function counterpart of an entropy condition that is classically used to single out physically relevant solutions to problem \eqref{mark}.
Equation \eqref{less regular} states that $[u]$ and $[\ns{\phi(u)}]$ are a very weak solution to problem \eqref{mark} (for the notion of very weak solution, see Lemma 5.3 of \cite{illposed}).

\begin{remark}\label{remark can't prove regularity}
	Numerical explorations of problem \eqref{pb discreto} suggest that its solution $u$ satisfies further regularity conditions. In particular, we conjecture that
	\begin{itemize}
		\item if $\phi'(x)>0$ for every $x\geq 0$, then for every initial data $\nu^{\ns{\phi(u)}}_{(t,x)}$ is Dirac for a.e.\ $(t,x) \in (0,+\infty) \times \Omega$;
		\item if $\ns\phi'(u(0,x))>0$ for a.e.\ $x\in\dom$, then $\nu^{\ns{\phi(u)}}_{(t,x)}$ is Dirac for a.e.\ $(t,x) \in (0,+\infty) \times \Omega$;
		\item for almost every initial data, there exists $\overline{t} \geq 0$ such that $\nu^{\ns{\phi(u)}}_{(t,x)}$ is Dirac for a.e.\ $(t,x) \in (\overline{t},+\infty) \times \Omega$ (this is a trivial consequence of the previous point and of the asymptotic analysis carried out in Section 6 of \cite{illposed});
		\item the $L^2$ norm of $\D^\pm \ns{\phi(u(t,\cdot))}$ is nonincreasing in time.
	\end{itemize}
	However, at this moment we only have been able to prove the first property. We sketch the proof of the first property and we discuss briefly the difficulties we encountered in the proof of the second. Denote by $\mu_L$ the Loeb measure induced by the product of the $\ns $Lebesgue measure over $\ns{[0,+\infty)}$ and the hyperfinite counting measure on $\Omega$. The desired property is a consequence of monotony of $\phi$ and of the fact that for almost every $x \in \dom$
	\begin{equation}\label{u_t a.e. finite}
	\lambda_L\left(\{ t \in \ns{[0,+\infty)} \cap \fin \} : u_t(t,x) \text{ is infinite}\} \right)= 0.
	\end{equation}
	Taking into account that $\norm{u}_\infty \in \fin$, in order to prove \eqref{u_t a.e. finite} it is sufficient to prove that $u(\cdot,x)$ does not have an infinite amount of oscillations in a finite time.
	The absence of such oscillations is a consequence of the smoothing properties of problem \eqref{pb discreto} under the hypothesis that $\phi'(x) >0$ for every $x \in \R$. In the context of nonstandard analysis, we can prove this assertion as follows.
	Let $l = \min_{x \in \dom} \ns{\phi(u(0,x))}$ and $L = \max_{x \in \dom} \ns{\phi(u(0,x))}$.
	By monotonicity of $\phi$, $u_t(t,x) > 0$ implies $0 < lu(t,x) \leq u_t(t,x) \leq Lu(t,x)$ and $u_t(t,x) < 0$ implies $Lu(t,x) \leq u_t(t,x) \leq lu(t,x)<0$.
	These estimates and the properties of the grid function formulation of the heat equation, discussed in Remark 5.6 of \cite{illposed}, are sufficient to conclude that for every $x \in \dom$ $u(\cdot, x)$ does not feature an infinite amount of oscillations in a finite time.
	
	Once we have shown that $u_t(t,x)$ is finite for a.e.\ finite $(t,x)$, we have that also $\Delta_\Lambda\ns{\phi}(u(t,x))$ is finite for a.e.\ finite $(t,x)$.
	By Corollary II.9 of \cite{watt}, for every $i \leq k$ and for a.e.\ finite $(t,x)$, if $(t',x') \sim (t,x)$ then $\D_i^- \ns{\phi(u(t,x))} \sim \D_i^- \ns{\phi(u(t',x'))}$. The hypothesis that $\phi$ is Lipschitz continuous entails that $\phi$ is a.e.\ of class $C^1$, i.e.\ that $\phi'$ is a.e.\ continuous over $\R$.
	This property and monotonicity of $\phi$ entail that for a.e.\ finite $(t,x)$, if $(t',x') \sim (t,x)$ then $u(t,x) \sim u(t',x')$. As a consequence $[u] \in C^0([0,+\infty), C^0(\Omega))$ and $\nu^{\ns{\phi}(u)}_{(t,x)} = \delta_{\phi([u](t,x))}$ for a.e.\ $(t,x) \in [0,+\infty) \times \Omega$.
	
	With a careful analysis, aided also by the study of the Riemann problem analysed in Section 7 of \cite{illposed}, it is similarly possible to prove that \eqref{u_t a.e. finite} is true also if if $\ns\phi'(u(0,x))>0$ for a.e.\ $x\in\dom$.
	Thus for a.e.\ finite $(t,x)$, if $(t',x') \sim (t,x)$ then $\D_i^- \ns{\phi(u(t,x))} \sim \D_i^- \ns{\phi(u(t',x'))}$.
	However, the argument based upon monotonicity of $\phi$ used in the previous case cannot be applied.
\end{remark}

In Section 5 of \cite{illposed} we have shown how equation \eqref{less regular} can be further sharpened under suitable hypotheses on the regularity of $\phi$ and of the solution $u$.
However notice that, under hypotheses (3) and (4), problem \eqref{mark} has different notions of solutions depending on the value of $u^+$:
\begin{itemize}
	\item if $u^+ < +\infty$, then the solution to problem \eqref{mark} is a Young measure that is the superposition of three Dirac measures centred at each branch of $\phi$;
	\item if $u^+ = +\infty$, then the solution to problem \eqref{mark} is the sum of a non-negative Radon measure and of a Young measure that is the superposition of two Dirac measures centred at each branch of $\phi$.
\end{itemize}
A priori, we expect that the classical solution to problem \eqref{mark} without these hypotheses still depends on the asymptotic behaviour of $\phi$. However, the analysis enabled by the grid function formulation will lead to a general definition of solution that is independent on the behaviour of $\phi$.
Nevertheless, in order to reach this goal we still need to discuss two different asymptotic behaviours of $\phi$: namely,
$\phi$ is eventually decreasing and $\phi$ is not eventually decreasing.
We will see that the latter case is the counterpart of hypothesis (3), while the former is a general counterpart of hypothesis (4). We start our analysis with the easiest of the two.

\subsection{$\phi$ is not eventually decreasing}
If $\phi$ is not eventually decreasing,
an analysis similar to the one carried out in Section 6 of \cite{illposed} leads to the conclusion that $\nu^u$ is a Young measure, i.e.\ that the solution to problem \eqref{pb discreto}
features only oscillations but no concentrations.

\begin{proposition}\label{prop sol (3)}
	If $\phi$ is not eventually decreasing, $\norm{u}_{\infty} \in \fin$.
\end{proposition}
\begin{proof}
	The desired bound over $\norm{u}_{\infty}$ can be obtained with an argument similar to that of point 2.\ of Proposition 4.6 of \cite{illposed}.
\end{proof}

Moreover, if $\phi^{-1}(r)$ is finite for every $r \in \R$ and if 
$\nu^{\ns{\phi(u)}}_{(t,x)}$ is Dirac for every $(t,x) \in (0,+\infty) \times \Omega$, $\nu^{u}_{(t,x)}$ can be decomposed as a sum of at most $\left|\phi^{-1}\left( [\phi(u)](t,x)\right) \right|$ Dirac measures.

\begin{corollary}\label{cor behav ym}
	Suppose that $\phi$ is not eventually decreasing, $\phi^{-1}(r)$ is finite for every $r \in \R$ and $\nu^{\ns{\phi(u)}}_{(t,x)}$ is Dirac for every $(t,x) \in (0,+\infty) \times \Omega$. Define $i : [0,+\infty) \times \Omega$ by $i(t,x) = \left|\phi^{-1}\left( [\phi(u)](t,x)\right) \right|$.
	Then there exist $\lambda_1, \ldots, \lambda_{i(t,x)}, r_1, \ldots, r_{i(t,x)} \in \R$ such that
	\begin{itemize}
		\item $\sum_{i = 1}^{i(t,x)} \lambda_i = 1$;
		\item $\phi(r_i) = \phi(r_j)$ for every $i, j \leq i(t,x)$;
		\item $\nu^{u}_{(t,x)} = \sum_{i = 1}^{i(t,x)} \lambda_i \delta_{r_i}$ for a.e.\ $(t,x) \in [0,+\infty) \times \Omega$.
	\end{itemize}
\end{corollary}

\subsection{$\phi$ is eventually decreasing}
Under the hypotheses that $\phi$ is eventually decreasing, taking into account also hypothesis (2) we have $\lim_{x \rightarrow +\infty} \phi(x) \in \R$.
In analogy to the case $u^+ = +\infty$ discussed in \cite{illposed}, if $\norm{u(0,\cdot)}_1$ is sufficiently large, then eventually $u(t,x)$ is infinite for some $x \in \dom$.
This property can be obtained from the discussion in Section 6.4 of \cite{illposed}, that does not rely on the hypothesis that $\lim_{x \rightarrow +\infty} \phi(x) = 0$, but only on the property that eventually $\phi'(x) <0$.
In other words, the solution to problem \eqref{pb discreto} features both oscillations and concentrations.
Thus $u$ can only be represented by the sum of a positive Radon measure $[u] - u_b$ and a Young measure $\nu^u$, as discussed in the comments to Theorem \ref{thm grid -> bounded}. 

If one further assumes that $\phi^{-1}(r)$ is finite for every $r \in \R$ and that $\nu^{\ns{\varphi(u)}}_{t,x}$ is Dirac for every $(t,x) \in (0,+\infty) \times \Omega$, the behaviour of the Young measure $\nu^u$ is analogous to the one described in Corollary \ref{cor behav ym}.

Despite this representation of the solution to problem \eqref{pb discreto}, there is no suitable notion of measure-valued solution for problem \eqref{mark}, since the composition between a continuous function $g$ and a Young measure is meaningful only under the hypothesis $g \in C^0_b(\R)$, i.e.\ $\lim_{|x| \rightarrow +\infty} g(x) = 0$.
In particular, we cannot apply Theorem \ref{parametrized measures} for the interpretation of the term $\ns{\varphi(u)}$, since $\lim_{x \rightarrow +\infty} \phi(x) \ne0$.
Instead, we need to generalize that result as follows.

\begin{proposition}\label{corollary parametrized measures} 
	For every $f \in \grid{\Omega}$, let $\nu^f$ be the parametrized measure satisfying Theorem \ref{parametrized measures}.
	If $g \in C^0(\R)$ satisfies $\lim_{|x| \rightarrow +\infty} g(x) = l \in \R$, then for all $\varphi \in C^0_c(\Omega)$
	\begin{equation*}\label{generalized young equivalence equation}
		\sh{\langle \ns{g}(f), \ns{\varphi} \rangle}
		=
		\int_{\Omega} \left( \int_{\R} g d \nu^f_x \right) \varphi(x) dx + \left(1-\nu^f(\R)\right)l \int_{\Omega} \varphi(x) dx.
	\end{equation*}
\end{proposition}
\begin{proof}
	If $g \in C^0(\R)$ satisfies $\lim_{|x| \rightarrow +\infty} g(x) = l$, then $\tilde{g}=g-l \in C^0_b(\R)$ and, by Theorem \ref{parametrized measures}, for all $\varphi \in C^0_c(\Omega)$
	$$
	\sh{\langle \ns{\tilde{g}}(f), \ns{\varphi} \rangle}
	=
	\int_{\Omega} \left( \int_{\R} \tilde{g} d \nu^f_x \right) \varphi(x) dx
	=
	\int_{\Omega} \left( \int_{\R} g - l d \nu^f_x \right) \varphi(x) dx
	$$
	As a consequence,
	\begin{align*}
	\sh{\langle \ns{g}(f), \ns{\varphi} \rangle}
	&=
	\sh{\langle \ns{\tilde{g}}(f)+l, \ns{\varphi} \rangle}\\
	&=
	\int_{\Omega} \left( \int_{\R} \tilde{g} d \nu^f_x \right) \varphi(x) dx + l \int_{\Omega} \varphi(x) dx.\\
	&=
	\int_{\Omega} \left( \int_{\R} \tilde{g}+l d \nu^f_x \right) \varphi(x) dx + \left(1-\nu^f(\R)\right)l \int_{\Omega} \varphi(x) dx.\\
	&=
	\int_{\Omega} \left( \int_{\R} g d \nu^f_x \right) \varphi(x) dx + \left(1-\nu^f(\R)\right)l \int_{\Omega} \varphi(x) dx.
	\end{align*}
\end{proof}

Proposition \ref{corollary parametrized measures} features some analogies with the generalized Young measures of DiPerna and Majda \cite{diperna}.
We believe that it can be suitably extended to vector-valued grid functions $f : \dom \rightarrow \hR^m$ and continuous functions $g : \R^m \rightarrow \R^n$ of a more general asymptotic behaviour.

\subsection{A general notion of measure-valued solution for problem \eqref{mark}}
Proposition \ref{corollary parametrized measures} enables a novel definition of measure-valued solution for problem \eqref{mark} under the very general hypotheses (1'), (2) and (5'). 
Notice that if $u$ is a solution to problem \eqref{pb discreto}, Proposition \ref{prop prop soluzione pb discreto} and Theorem \ref{parametrized measures} ensure that $\nu^u$ is a Young measure, i.e.\ $\nu^u(\R) = 1$. Hence, according to Proposition \ref{corollary parametrized measures} the contribution of $\ns{\phi(u)}$ when $u$ is infinite is negligible.
Thus we get the following definition of measure-valued solution of problem \eqref{mark}.

\begin{definition}\label{def sol}
	An entropy measure-valued solution of problem \eqref{mark} consists of a Young measure $\nu$ over $[0,T]\times\Omega$
	and of a positive Radon measure $\mu \in \rad([0,T]\times\Omega)$, satisfying the conditions:
	\begin{enumerate}
		\item the barycentre $b(t,x)$ of $\nu$ satisfies $b \in L^1([0,T]\times\Omega)$;
		\item the function $v(t,x) = \int_{\R} \phi(\tau) d\nu_{(t,x)}$ satisfies $v \in L^\infty([0,T]\times\Omega) \cap L^2([0,T], H^1(\Omega))$;
		\item $(b+\mu)_t = \Delta v$ in the the sense that
		\begin{equation}\label{soluzione smarrazzo}
		\int_{0}^T \langle \mu, \varphi_t\rangle_{C^0(\Omega)} dt + \int_0^T \int_{\Omega} b \varphi_t - \nabla v \cdot \nabla \varphi dx dt + \int_{\Omega} u_0(x) \varphi(0,x) dx = 0,
		\end{equation}
		for all $\varphi \in C^1([0,T] \times \overline{\Omega})$ with $\varphi(T,x) = 0$ for all $x\in\Omega$;
		\item for all $g \in C^1(\R)$ with $g' \geq 0$, define
		$$
		G(x) = \int_0^x g(\phi(\tau)) d\tau \text{ and } G^\star(\nu)= \int_{\R} G(\tau) d\nu.
		$$
		Then $\nu$ and $v$ satisfy the entropy inequality
		\begin{equation}\label{entropy inequality}
		\int_0^T\int_{\Omega} G^\star(\nu) \varphi_t - g(v)\nabla \cdot v \nabla \varphi -g'(v)|\nabla v|^2\varphi dx dt \geq 0
		\end{equation}
		for all $\varphi \in \tests([0,T] \times \Omega)$ with $\varphi(t,x) \geq 0$ for all $(t,x) \in [0,T] \times \Omega$.
	\end{enumerate}
	The solution is global if in the above formulas we can replace the interval $[0,T]$ with $[0,+\infty)$.
\end{definition}

\begin{remark}
	The notion of solution defined by Plotnikov in \cite{plotnikov} under hypothesis (3) and the the notion of solution defined by Smarrazzio in \cite{smarrazzo} under hypothesis (4) can both be recovered by defining $u(t,x) = b(t,x)$ and $v(t,x) = \int_{\R} \phi(\tau) d\nu_{(t,x)}$.
\end{remark}

As expected, if the solution to the grid function formulation is regular enough then it is an entropy measure-valued solution to problem \eqref{mark} in the sense of the above definition.

\begin{proposition}\label{prop se suff regolare allora soluzione}
	Let $u$ be the solution to the grid function formulation \eqref{pb discreto}.
	If $\nu^{\ns{\phi(u)}}_{(t,x)}$ is Dirac for a.e.\ $(t,x) \in (0,+\infty) \times \Omega$ and its barycentre $v(t,x) = \int_{\R} \tau d \nu^{\ns{\phi(u)}}_{(t,x)}$ satisfies $v \in L^2((0,+\infty), H^1(\Omega))$, then problem \eqref{mark} has a global entropy measure-valued solution in the sense of Definition \ref{def sol}.
\end{proposition}
\begin{proof}
	If $u$ is a solution to the grid function formulation \eqref{pb discreto} satisfying the hypotheses, we claim that $\nu = \nu^u$ and $\mu = [u]-b$ are an entropy measure-valued solution of problem \eqref{mark}.
	%
	
	(0) Notice that Corollary \ref{cor parametrized measures} entails that $v(t,x) = \int_{\R} \phi(\tau) d\nu^u_{(t,x)} = \int_{\R} \tau d\nu^{\ns{\phi(u)}}_{(t,x)}$.
	
	(1) $b \in L^1([0,+\infty)\times\Omega)$ by hypothesis (5'), by Proposition \ref{prop prop soluzione pb discreto} and by the fact that $\norm{b}_1 \leq \norm{u}_1$. The latter inequality is a consequence of Proposition 4.3 of \cite{ema2}.
	
	(2) The fact that $v \in  L^\infty([0,+\infty)\times\Omega)$ is a consequence of the estimate $\norm{\ns{\phi(u)}}_\infty \in \fin$, argued in the proof of Proposition \ref{prop prop soluzione pb discreto}. Moreover, $v \in L^2([0,+\infty), H^1(\Omega))$ is one of our hypotheses.
	
	(3) The validity of \eqref{soluzione smarrazzo} is a consequence of \eqref{less regular} and of our hypotheses on the regularity of $u$.
	The proof is analogous to that of Theorem 5.7 of \cite{illposed}.
	
	(4) The validity of the entropy inequality \eqref{entropy inequality} is a consequence of \eqref{eq entropy} and of our hypotheses on the regularity of $u$.
	The proof is analogous to that of point 1.\ of Theorem 5.4 of \cite{illposed}.
\end{proof}

In order to show that problem \eqref{mark} has an entropy measure-valued solution in the sense of Definition \ref{def sol} for a particular choice of $\phi$, it is possible to show that the solution to the corresponding grid function formulation \eqref{pb discreto} is regular enough to satisfy the hypotheses of Proposition \ref{prop se suff regolare allora soluzione}.
As we have acknowledged in Remark \ref{remark can't prove regularity}, we suspect that this is indeed the case for a large class of initial data, but we have not been successful in proving this conjecture.

\subsection{The correspondence between the grid function formulation and a sequence of approximating problems}\label{sec applicazione thm}
By Theorem \ref{thm grid -> bounded}, the measure-valued solution to problem \eqref{mark} induced by a solution $u$ of the grid function formulation can be obtained as the limit of a bounded sequence in $L^1(\Omega)$. This leads to the conjecture that such a measure-valued solution corresponds to the solution obtained via a sequence of well-posed approximating problems. Previous works by Plotnikov \cite{plotnikov} and Smarrazzo \cite{smarrazzo} on problem \eqref{mark} and the validity of the nonstandard entropy estimate \eqref{eq entropy} suggest that such approximating problems might be the pseudoparabolic regularizations
\begin{equation*}\label{eqn pb regolarizzato}
	\left\{
	\begin{array}{l}
	\partial_t u=\Delta \phi(u) + \eta \Delta u_t \text{ in } \Omega\\
	\frac{\partial \phi(u) + \eta u_t}{\partial \hat{n}} = 0 \text{ in } [0,T] \times \partial\Omega\\
	u(0,x) = u_0(x)
	\end{array}\right.
\end{equation*}
with $\eta \in \R$, $\eta > 0$.

If $\phi$ is not eventually decreasing, the classical counterparts to
the general existence result provided by our Proposition \ref{prop prop soluzione pb discreto} and Proposition \ref{prop se suff regolare allora soluzione}
might be obtained by adapting the techniques of \cite{plotnikov}.
Instead, if $\phi$ is eventually decreasing, the desired results can be obtained by adapting the argument of Section 2 of \cite{smarrazzo} under the hypothesis that $u^+ = +\infty$. However, one has to take into account that $\lim_{x \rightarrow +\infty} \phi(x)$ might be positive and adapt the corresponding limiting arguments in a suitable way.

This example suggests another interpretation of Theorem \ref{thm grid -> bounded}: if the grid function formulation of a PDE has a solution in $\mathbb{L}^1(\Omega)$, then the grid function formulation corresponds to a sequence of classical regularized problems.
Both the grid function formulation and the classical regularization have advantages and disadvantages. As we have seen in this section, the grid function formulation allows to easily obtain existence results of very weak solutions for a broad class of problems that classically must be approached with different techniques. The unifying nature of this approach is worthwhile on its own; moreover, the grid function formulation enabled a uniform definition of solution that does not depend upon additional hypotheses (in the case of problem \eqref{mark}, these additional hypotheses are the ones regarding the behaviour of $\phi$).
Finally, as discussed in \cite{illposed}, the hyperfinite discretization in space enables the study of the asymptotic behaviour with techniques from dynamical systems. A drawback of the grid function formulation is that currently it seems harder to provide sharp results on the regularity of the grid solutions. However, this problem might be caused by a weakness of the author rather than by a flaw of the approach.

Conversely, the use of approximating problems for the study of problem \eqref{mark} requires different techniques that depend upon the asymptotic behaviour of $\phi$. This has caused a delay of almost fifteen years between the discussion of the cases $u^+ < +\infty$ and $u^+ = +\infty$; moreover, more general hypotheses over $\phi$ have not yet be studied. However, for problem \eqref{mark} it appears to be easier to prove that the solution obtained via some approximating problems is regular enough to satisfy Definition \ref{def sol}.

Such interplay between the classic techniques of analysis of PDEs and an approach based on a grid function formulation, enabled by Theorem \ref{thm grid -> bounded}, might allow for a combined strategy for the study of ill-posed PDEs that exploits the strengths of each approach.

Further applications of grid functions to PDEs will be discussed in \cite{forthcoming}.

\textbf{Acknowledgements:} we are grateful to an anonymous referee for their observations and suggestions.

\end{document}